\documentclass[12pt,a4paper]{article}
\usepackage{amsmath,amsthm,amscd,amsfonts,amssymb,enumitem}
\usepackage{latexsym}
\usepackage{multirow}
\usepackage{xcolor, soul}
\usepackage{caption,float}

\usepackage[colorlinks=true, linkcolor=blue, anchorcolor=blue, citecolor=blue, filecolor=blue, urlcolor=blue]{hyperref}
\newtheorem{thm}{Theorem}[section]
\newtheorem{lemma}[thm]{Lemma}

\theoremstyle{definition}
\newtheorem{defin}[thm]{Definition}

\newtheorem{exam}[thm]{Example}
\theoremstyle{remark}
\newtheorem{rem}[thm]{Remark}

\newcommand\numberthis{\addtocounter{equation}{1}\tag{\theequation}}

\numberwithin{equation}{section}

\def\CC{\mathbb C}
\def\NN{\mathbb N}
\def\ZZ{\mathbb Z}
\def\RR{\mathbb R}

\def\bb{\mathcal B}
\def\nn{\mathcal N}

\def\sinc{\text{sinc}\,}
\def\lin{\text{lin}}

\title{Random Sampling of Mellin Band-limited Signals}
\author{Shivam Bajpeyi\thanks{Email id: shivambajpai1010@gmail.com} \ Dhiraj Patel\thanks{Email id: dpatel.iitd@gmail.com} \ S. Sivananthan \thanks{Email id: siva@maths.iitd.ac.in}\\
	Department of Mathematics, Indian Institute of Technology Delhi,\\ New Delhi-110016, India}

\date{}
\begin{document}
	
	\maketitle
	
	\begin{abstract}
		 In this paper, we address the random sampling problem for the class of Mellin band-limited functions $\bb_T$ which is concentrated on a bounded cube. It is established that any function in $\bb_T$ can be approximated by an element in a finite-dimensional subspace of $\bb_T.$ Utilizing the notion of covering number and Bernstein's inequality to the sum of independent random variables, we prove that the random sampling inequality holds with an overwhelming probability provided the sampling size is large enough.
		\\
		\vskip0.001in
		
	\end{abstract}

	\section{Introduction}
	The theory of sampling and reconstruction has been receiving significant attention from the broad areas of approximation theory, signal processing, and digital communication. The problem of sampling 
	for the function space $V$ is mainly concerned with the stable recovery of a function $f \in V$ from its sample values $\{f(x_{j})\}$ on some countable set $X=\{x_{j}:j \in J \}$, without losing any information. The sampling set $X$ is known as a stable sampling set. For the space of Mellin band-limited functions $\bb_T$ (defined in Section \ref{section2}), the sampling problem is equivalent to finding a sampling set $\{x_j:j\in J\}$ such that the sampling inequality
	\begin{equation} \label{main inq}
		A \|f\|_{X_c^{2}(\RR^{n}_{+})}^2 \leq \sum_{j \in J} |f(x_j)|^2 x_{j}^{2c} \leq B \|f\|_{X_c^{2}(\RR^{n}_{+})}^2,
	\end{equation}
	holds for all $f$ in $\bb_T$, for some $A,B>0$. Here $\RR_{+}$ denotes the set of positive real numbers. The Banach space $X_c^2(\mathbb{R}^{n}_{+})$ is a collection of all measurable function $f$ with $$\|f\|_{X_c^{2}(\RR^{n}_{+})}^2=\int_{\RR_{+}^{n}} |x|^{2c-1}|f(x)|^2\, dx<\infty.$$
	An important result in this direction was collectively given by Bertero and Pike \cite{bertero}, and Gori \cite{gori} which asserts that any function $f$ in $\bb_T$ can be stably recovered from its sample at exponentially spaced sample points $\{ e^{\frac{k}{T}}:k\in \ZZ \}$ and the reconstruction formula is given by
	\begin{equation} \label{expsmpl}
		f(x)= \sum_{k \in \ZZ} f(e^{\frac{k}{T}})\, \lin_{\frac{c}{T}}(e^{-k}x^{T}), \qquad x\in \RR_{+},
	\end{equation}
	where $\lin_{c}(x) = x^{-c}\sinc(\log x)$ and $\sinc(u)= \frac{\sin \pi u}{\pi u}.$ In this case, the sampling set $\{e^{k/T}:k\in \ZZ\}$ is a stable sample set for $\bb_T$ and $A=B=T.$ The exponentially spaced sampling set is useful in the application where the signal information is accumulated near zero. The exponential sampling method is also useful to deal with various real-world problems arising in science and engineering (see \cite{casasent,ostrowsky}). Butzer and Jansche \cite{butzer5} pioneered the mathematical study of the exponential sampling formula \eqref{expsmpl} employing the tools of the theory of Mellin transform. Some recent developments related to the exponential sampling problem can be observed in \cite{ownkant,nfo,bardaro7,mleb}.\par
	
	The signals are used to recover from their random measurements in the field of image processing \cite{chan}, compressed sensing \cite{candes,eldar}, and machine learning \cite{cucker1,cover}. The random sampling problem deals with the following question. What is the probability that the uniformly distributed random sample points satisfy the sampling inequality \eqref{main inq} for some class of functions in $\bb_T$? Due to the randomness of the samples, the sampling inequality \eqref{main inq} may not hold surely for $\bb_T$. So, the random sampling problem estimates the following probability
	$$ P \Big( A \|f\|_{X_c^{2}(\RR^{n}_{+})}^2 \leq \sum_{j \in J} |f(x_j)|^2 x_{j}^{2c} \leq B \|f\|_{X_c^{2}(\RR^{n}_{+})}^2 \Big) \geq 1-\epsilon,$$ for some class of Mellin band-limited function, where $\epsilon >0$ is arbitrary small.\par
	 
In the case of finite dimensional space, the random sampling problem was studied for the space of multivariate trigonometric polynomials \cite{2005,xian}. For the infinite-dimensional space of Fourier band-limited functions, the stable recovery from their random samples scattered throughout $\RR^n$ was shown to be impossible \cite{groch}. This motivated to consider the random samples distributed on a bounded cube $C\subset \RR^n$ and the class of functions concentrated on $C$. It was established in \cite{groch,groch1} that the class of Fourier band-limited functions concentrated on $C$ can be stably recovered from its samples at uniformly distributed random points on $C$. The random sampling problem has been investigated for the shift-invariant spaces in \cite{devrandom,fuhr,yang}. In recent years, the related random sampling problem has been studied on the sphere in \cite{pams}, for the image space of an idempotent integral operator \cite{sun2021,jma,rks3} which generalizes the space of shift-invariant space and the signal space of finite rate of innovation, see \cite{sun2010}.	\par
 
    A Mellin band-limited function can not be Fourier band-limited at the same time, and it is analytically extensible to the Riemann surface of the (complex) logarithm (see \cite{bardaro4}). Since one must appropriately extend the idea of the Bernstein spaces, involving Riemann surfaces, the theory of Mellin band-limited functions differs greatly from the theory of Fourier band-limited functions.
	The main aim of the paper is to study the stability of the random sampling set for the space Mellin bandlimited functions, using the Mellin transform and its inversion theory. Motivated from \cite{groch}, we consider the problem of random sampling for the class of $\delta$-concentrated Mellin band-limited functions on $C_{R}:=[1/R, R]^n\subset \RR_{+}^n$, i.e.,
	$$\bb_{T,\delta}:=\Big\{f \in \bb_{T}:  \int_{C_R} |f(x)|^2 x^{2c-1} dx \geq (1-\delta)\|f\|^{2}_{X_{c}^{2}(\RR_{+}^n)} \Big\},$$
	where $\delta\in (0,1),$ $R>1,$ and $c \in \mathbb{R}^n.$ 
	\par 
	The proposed plan of the paper goes as follows. Section \ref{section2} provides some preliminaries required to derive the desired  results. In Section \ref{section3}, we first establish that any element in $\bb_{T}$ can be approximated by an element in a finite-dimensional subspace of $\mathcal{B}_{T}$ with the help of the representation formula \eqref{expsmpl}. Then we show that any bounded subset of $\mathcal{B}_{T,\delta}$ is totally bounded with respect to $\|\cdot \|_{X_{c}^{\infty}(C_R)}$. In the end, by using the notion of covering number and the well-known Bernstein's inequality, we prove that the random sampling inequality holds with overwhelming probability for the functions concentrated on the compact set.
	
	\section{Preliminaries}\label{section2}
	In this section, we define some preliminary results, and basic notations used in the rest of the paper are given in Table 1.
	\par 
	For $1\leq p<\infty$, let $L^{p}(\RR^{n}_{+})$ be the space of $p-$integrable functions in the Lebesgue sense on $\mathbb{R}^{n}_{+}$ with usual $p-$norm. Moreover, $L^{\infty}(\RR^{n}_{+})$ denotes the class of bounded measurable functions defined on $\mathbb{R}^{n}_{+}$ with norm $ \|f\|_{L^{\infty}(\RR^{n}_{+})} := ess\sup_{x \in \mathbb{R}^{n}_{+}}|f(x)|.$ Let $c \in \RR^n$ be fixed and for $1 \leq p \leq \infty,$ we define the space
	$$X_{c}^{p}(\mathbb{R}^{n}_{+}):=\big\{f:\mathbb{R}^{n}_{+} \rightarrow \mathbb{C}: (\cdot)^{c-1/p}f(\cdot) \in L^{p}(\mathbb{R}^{n}_{+})\big\}$$ equipped with norm
	$$\|f\|_{X_{c}^{p}(\mathbb{R}^{n}_{+})}= \|(\cdot)^{c-1/p} f(\cdot) \|_{L^{p}(\mathbb{R}^{n}_{+})}.$$ It is important to mention that $X_{c}^{p}(\mathbb{R}^{n}_{+})$ is Banach space (see \cite{butzer7}). Moreover, for $p=2$, $X_{c}^{2}(\mathbb{R}^{n}_{+})$ is a Hilbert space with the inner product
	$$ \langle f,g \rangle_{X_{c}^{2}(\mathbb{R}^{n}_{+})} := \int_{\RR_{+}^{n}} f(x) \overline{g(x)} x^{2c-1} dx.$$
	Let $c + it =:s \in \CC^n.$ Then the multi-dimensional Mellin transform of any function $f \in X_{c}^{1}(\mathbb{R}^{n}_{+})$ is given by (see \cite{tuan})
	$$\hat{M}[f](s) :=\int_{\mathbb{R}^{n}_{+}} f(x)\ x^{s-1}\ dx.$$ For a fixed $c \in \RR^n,$ the corresponding inverse Mellin transform is defined as
	$$\hat{M}^{-1}[f](x) := \frac{1}{(2 \pi i)^n}\int_{c+i \mathbb{R}^{n}} \hat{M}[f](s) \ x^{-s}\ ds.$$
    
    Riemann first employed Mellin transform in his renowned memoir on the Riemann zeta function. Later, Mamedov studied the Mellin transform and its properties in \cite{mamedeo}. The article \cite{butzer3} presents various theoretical aspects of Mellin transform along with its applications to different areas of mathematical analysis. For some significant development of the Mellin theory, we also refer to \cite{bardaro1,bardaro2,bardaro3,bardaro4}.\par
	
    \begin{defin}
		For $c,t \in \mathbb{R}^n$ and $T>0,$ any function $f \in X_{c}^{2}(\mathbb{R}^{n}_{+})$ is said to be Mellin band-limited to $[-T,T]^n$ if $\hat{M}[f](c+it)=0$ for all $\|t\|_{\infty}> T.$ We denote $\mathcal{B}_{T}$ the space of all Mellin band-limited functions, i.e.,
		$$ \bb_{T} = \big\{f \in X_{c}^{2}(\mathbb{R}^{n}_{+}): \hat{M}[f](c+it)=0 \ \mbox{ for \ all}\ \|t\|_{\infty} >T \big\}.$$
		
		\begin{table}[H] \label{table1}
			\centering
			\begin{tabular}{|c|c|}
				\hline
				Notation & Remark\\
				\hline
				$x:=\big(x(1),x(2),\dots,x(n)\big)\in \RR^n$ & $x(i)\in \RR$\\
				\hline
				$\frac{x}{y}:=\Big(\frac{x(1)}{y(1)},\frac{x(2)}{y(2)},\dots,\frac{x(n)}{y(n)}\Big)$ & $x,y\in \RR_{+}^n$\\
				\hline
				$\log x:=\big(\log x(1),\log x(2),\dots,\log x(n)\big)$ & $x\in \RR_{+}^n$\\
				\hline
				$x^c:=x(1)^{c(1)}\cdot x(2)^{c(2)}\cdots x(n)^{c(n)}$ & $x\in \RR_{+}^n$, $c\in \RR^n$\\
				\hline
				$\alpha^x:=\big( \alpha^{x(1)},\alpha^{x(2)},\dots,\alpha^{x(n)} \big)$ & $\alpha\in \RR_{+}$, $x\in \RR^n$\\
				\hline
				$\sinc(x):=\frac{\sin \pi x(1)}{\pi x(1)}\cdot \frac{\sin \pi x(2)}{\pi x(2)}\cdots \frac{\sin \pi x(n)}{\pi x(n)}$ & $x\in \RR^n$\\
				\hline
				$\lin_{c}(x):=x^{-c}\sinc(\log x)$ & $x\in \RR_{+}^n$\\
				\hline
			\end{tabular}
			\caption{Some notations used for the rest of the paper}
		\end{table}
		
	\end{defin}
	A Hilbert space $H$ of functions defined on $\Lambda$ is a reproducing kernel Hilbert space if for each $x\in \Lambda$, the point evaluation functional $f \mapsto f(x)$ is continuous, i.e., for each $x \in \Lambda,$ there exists positive constant $C_{x}$ such that
	$$|f(x)| \leq C_{x} \|f\| \qquad \forall\, f \in H.$$
	\begin{lemma} \label{estmt}
		The space $\bb_T$ is a reproducing kernel space. Moreover, for $x \in \mathbb{R}^{n}_{+},$ we have
		$$ | f(x)| \leq  x^{-c} \| f\|_{X_{c}^{2}(\mathbb{R}^{n}_{+})}\ ,\ \ \forall f \in \bb_T.$$
	\end{lemma}
	
	\begin{proof}
		Let $f$ be Mellin band-limited to $[-T,T]$, and $T>0.$ Then from \cite[Theorem 4]{bardaro1}, we have
		\begin{equation} \label{rkp1}
			f(x)= T \int_{\RR_{+}} f(y)\, \lin_{\frac{c}{T}} \big((x/y)^T \big) \frac{dy}{y}, \qquad x\in \RR_{+}.
		\end{equation}
		Using $n-$dimensional Mellin transform and inverse Mellin transform, we extend (\ref{rkp1}) for $n$-dimensional Mellin band-limited function. In particular, we have
		\begin{equation} \label{rkp}
			f(x)= T^n \int_{\RR_{+}^n} f(y)\, \lin_{\frac{c}{T}} \Big(\frac{x}{y} \Big)^T y^{-1}dy,
		\end{equation}
		where $\lin_{\frac{c}{T}} \big(\frac{x}{y} \big)^T:=x^{-c}y^c \ \sinc(T\log x-T\log y).$ In view of (\ref{rkp}) and H\"{o}lder's inequality, we have
		\begin{eqnarray*}
			|f(x)| & \leq & T^n x^{-c} \left( \int_{\mathbb{R}^{n}_{+}} |f(y)|^2 y^{2c-1} dy \right)^{1/2} \left( \int_{\mathbb{R}^{n}_{+}} \sinc^2 \left(T \log(x/y) \right) y^{-1} dy \right)^{1/2} \\
			& \leq & x^{-c} \|f\|_{X_{c}^{2}(\RR_{+}^n)} \|\sinc\|_{L^2(\mathbb{R}^{n})}\\
			& \leq & x^{-c} \|f\|_{X_{c}^{2}(\RR_{+}^n)}.
		\end{eqnarray*}
		This completes the proof.
	\end{proof}
	Now we see an example of Mellin band-limited function which is important due to its major application in approximation theory. For instance, it plays a vital role in the study of the convergence of generalized exponential sampling series, see \cite{bardaro7}.
	\begin{exam} \cite{bardaro7}
		Let $c\in \RR, \alpha \geq 1,$ and $k \in \mathbb{N}$. Then the \textit{Mellin Jackson kernel} is defined as
		$$J_{\alpha,k}(x):= C_{\alpha,k}\ x^{-c} \sinc^{2k} \left(\frac{\log x}{2 \alpha k \pi} \right),$$
		where $x\in \RR_{+}$ and $\displaystyle C^{-1}_{\alpha,k} := \int_{0}^{\infty} \sinc^{2k} \left(\frac{\log x}{2 \alpha k \pi} \right)\frac{dx}{x}.$ Since $\hat{M}[J_{\alpha,k}](c+it)=0$ for $|t| \geq \frac{1}{\alpha},$ the function $J_{\alpha,\eta}$ is Mellin band-limited. In particular, for $k=1,$ this is well-known as \textit{Mellin Fejer kernel}, is given by
		$$F_{\rho}^{c}(x)= \frac{\rho}{2\pi}x^{-c}\sinc^{2} \left( \frac{\rho \log \sqrt{x}}{\pi} \right).$$
	\end{exam}
	
	\section{Main Results} \label{section3}
	In this section, we discuss the proposed results of the paper. We see that for each $f\in \bb_{T},$ the following representation holds:
	$$ f(x)= \sum_{k \in \ZZ^n} f(e^{k/T})\,\lin_{\frac{c}{T}}(e^{-k}x^{T}), \qquad x\in \RR_{+}^n.$$
	Now we aim to show that any $f \in \mathcal{B}_{T}$ can be approximated on $C_R$ by an element from a finite-dimensional subspace $\mathcal{B}^{N}_{T}$ of $\mathcal{B}_{T}$ where $\mathcal{B}^{N}_{T}$ is defined by
	$$\bb_{T}^N= \Big\{ \sum_{k \in \ZZ^n \cap [-\frac{N}{2},\frac{N}{2}]^n} c_k\, \lin_{\frac{c}{T}} \big(e^{-k} (\cdot)\big): c_k\in \RR \Big\}.$$
	\begin{thm} \label{fd}
		Let $\epsilon >0.$ Then for each $f\in \bb_{T},$ there exists $f_{N} \in \bb_{T}^N$ such that
		$$ \|f-f_{N} \|_{X_{c}^{\infty}(C_{R})} < \epsilon \|f\|_{X_c^2(\mathbb{R}^{n}_{+})},$$ 
		whenever $ N > 4T \pi^{-2} \epsilon^{-2/n}+ 2 T \log R.$
	\end{thm}
	\begin{proof} For any $f \in \mathcal{B}_{T}$ and $x \in C_{R},$ we have $f_N \in \mathcal{B}^{N}_{T}$ as
		$$f_{N}(x)= \sum_{k \in \ZZ^n \cap [-\frac{N}{2},\frac{N}{2}]^n} f(e^{k/T})\, \lin_{\frac{c}{T}} (e^{-k}x^{T}).$$
		In view of Cauchy-Schwarz inequality, we obtain
		\begin{align*}
			&x^c \big|f(x)-f_{N}(x) \big|\\
			=& \bigg| x^c \sum_{k \in \mathbb{Z}^n \setminus[-\frac{N}{2},\frac{N}{2}]^n}f(e^{k/T})\, \lin_{\frac{c}{T}} (e^{-k}x^{T}) \bigg| \\
			=& \bigg| \sum_{k \in \mathbb{Z}^n \setminus[-\frac{N}{2},\frac{N}{2}]^n} f(e^{k/T}) \ e^{\frac{\langle c,k \rangle_2}{T}}\, \sinc(T \log x-k) \bigg| \\
			\leq & \bigg(\sum_{k \in \mathbb{Z}^n \setminus[-\frac{N}{2},\frac{N}{2}]^n} |f(e^{k/T})|^{2} e^{\frac{2\langle c,k \rangle_2}{T}}  \bigg)^{1/2} \bigg(\sum_{k \in \mathbb{Z}^n \setminus[-\frac{N}{2},\frac{N}{2}]^n} \sinc^{2}(T \log x - k) \bigg)^{1/2}. \numberthis \label{est}
		\end{align*}
		
		In view of \cite[Theorem 4]{bardaro1}, for $j,k \in \ZZ^n,$ we have 
		$$ T^n e^{\frac{-2\langle c,k \rangle_2}{T}} \big\langle \lin_{\frac{c}{T}}(e^{-k} x^T),\lin_{\frac{c}{T}}(e^{-j} x^T) \big\rangle_{X_c^2(\mathbb{R}^{n}_{+})} = \lin_{\frac{c}{T}}(e^{j-k})= \delta_{j,k}\ ,$$ where $\delta_{j,k}=$
$    \begin{cases}
            1, &         \text{if } j=k,\\
            0, &         \text{if } j\neq k.
    \end{cases} $
    
    This gives
		\begin{eqnarray*}
			\|f\|_{X_{c}^{2}(\mathbb{R}^{n}_{+})}^{2} &=& \sum_{k \in \ZZ^n} \sum_{j \in \ZZ^n} f(e^{k/T}) \overline{f(e^{j/T})}\, \big\langle  \lin_{\frac{c}{T}} (e^{-k}x^{T}), \lin_{\frac{c}{T}} (e^{-j}x^{T}) \big\rangle_{X_{c}^{2}(\mathbb{R}^{n}_{+})} \\
			&=& \frac{1}{T^n} \sum_{k \in \ZZ^n} |f(e^{k/T})|^{2}  e^{\frac{2\langle c,k \rangle_2}{T}}.
		\end{eqnarray*}
		Now for $\displaystyle x \in C_{R},$ we have
		\begin{align*}
			\sum_{k\in \ZZ^n \setminus[-\frac{N}{2},\frac{N}{2}]^n} \sinc^{2} \pi(T \log x - k)\leq & \pi^{-2n} \sum_{k \in \ZZ^n \setminus[-\frac{N}{2},\frac{N}{2}]^n} \left(\prod_{i=1}^{n} \sin^{2}(T\log x_i-k_i) \,(T \log x_i - k_i)^{-2} \right) \\
			\leq & \pi^{-2n}\int_{\RR^n \setminus[-\frac{N}{2},\frac{N}{2}]^n}  \left( \prod_{i=1}^{n} (T \log x_i - y_i)^{-2} \right) dy \\
			\leq & \pi^{-2n}\int_{\RR^n \setminus[-\frac{N}{2}+T \log R,\frac{N}{2}-T \log R]^n} y^{-2}\,dy \\
			=& \frac{{4^n}}{\pi^{2n}(N-2T\log R)^n}.
		\end{align*}
		On combining these estimates, we obtain from (\ref{est}) that
		$$ \|f-f_{N} \|_{X_{c}^{\infty}(C_{R})} \leq \left(T^{n/2}  \|f\|_{X_{c}^{2}(\mathbb{R}^{n}_{+})} \right) \left(\frac{4}{\pi^{2}(N-2T\log R)} \right)^{n/2}.$$
  This implies that if  $ N > 4T \pi^{-2} \epsilon^{-2/n}+ 2 T \log R,$ then $$\|f-f_{N} \|_{X_{c}^{\infty}(C_{R})} < \epsilon \|f\|_{X_c^2(\mathbb{R}^{n}_{+})}.$$
		\color{black}
		This completes the proof.
	\end{proof} 
	
	In order to derive the desired probability estimate for the sampling inequality \eqref{main inq} to hold, we first prove that the bounded subset of $\delta$-concentrated functions is totally bounded. 
	\begin{lemma} \label{tb}
		The set $\bb_{T,\delta}^{*}:= \big\{ f\in \bb_{T,\delta}: \|f\|_{X_{c}^{2}(\RR_{+}^n)}=1 \big\}$ is totally bounded with respect to $\|\cdot\|_{X_{c}^{\infty}(C_{R})}.$
	\end{lemma}
	
	\begin{proof}
		It follows similar arguments as in \cite[Lemma 2.4]{jma}. By Theorem \ref{fd}, for given $\epsilon>0$ and $f \in \mathcal{B}^{*}_{T,\delta},$ there exists $f_N \in \mathcal{B}^{N}_{T}$ such that $\|f-f_N \|_{X_{c}^{\infty}(C_R)} < \frac{\epsilon}{2}.$ Let $\overline{B\left(0;\epsilon \right)}$ represents the closed ball in $\mathcal{B}^{N}_{T}$ having radius $\epsilon$ and centred at origin. Now from Lemma \ref{estmt}, we have $\| f\|_{X_{c}^{\infty}(C_{R})} \leq  1$ which implies that $f_N \in \overline{B\left(0;1+\frac{\epsilon}{2} \right)}.$ Since $\overline{B\left(0;1+\frac{\epsilon}{2} \right)}$ is totally bounded, the desired result follows. 
	\end{proof}
	\begin{defin}
		The covering number of any bounded subset $Y$ of a Banach space $X$ with respect to $\beta$ is defined by
		$$\nn(Y,\beta):=\min \Big\{\ell \in \mathbb{N}: \exists\  a_1,a_2,...,a_\ell \in X\ \mbox{such that}\ Y \subset \bigcup_{m=1}^{\ell} B(a_{m};\beta) \Big\},$$ where $B(a_{m};\beta)$ represents open ball of radius $\beta$ and centered at $a_m$ in $X.$ 
	\end{defin}
	The following lemma provides a bound for the covering number of any closed ball in a finite-dimensional Banach space.
	\begin{lemma} \cite{cover} \label{cover}
		Let X be a Banach space of dimension $d$ and $\overline{B(0;s)}$ represents the closed ball of radius $s$ centered at the origin. Then the minimum number of open balls of radius $r$ required to cover $\overline{B(0;s)}$ is bounded by $\displaystyle \Big(\frac{4s}{r} \Big)^d.$
	\end{lemma}
	\begin{rem}
		From Theorem \ref{fd} and Lemma \ref{tb}, for $f \in \mathcal{B}^{*}_{T,\delta},$ it can be seen that for $\|f-f_N \|_{X_{c}^{\infty}(C_R)} < \frac{\epsilon}{2},$ we need to choose $N$ in such a way that
		$$ N >  \left[4T \pi^{-2} 2^{\frac{n}{2}} \epsilon^{-2/n}\right]+ 2 T \log R.$$ Particularly, if we set $\displaystyle N = 4T \pi^{-2} 2^{\frac{n}{2}} \epsilon^{-2/n} + 2 T \log R+1,$ then the dimension of $B_{N}^{T}$ will be bounded by
		\begin{eqnarray*}
			N^n &=& \left[T 2^{\frac{n}{2}+2} \pi^{-2} \epsilon^{-2/n} + 2 T \log R +1 \right]^n \\
			&\leq & 2^n \left[\frac{T^n 2^{n\left(\frac{n}{2}+2  \right)}\pi^{-2n}}{\epsilon^2}+ (2 T \log R+1)^n \right]:= d_{\epsilon}.
		\end{eqnarray*}
		\color{black}
		Now from Lemma \ref{tb} and \ref{cover}, the bound for the covering number of $\bb_{T,\delta}^{*}$ is given by  
		\begin{equation} \label{covering}
			\nn\left(\bb_{T,\delta}^{*},\epsilon \right) \leq M \bigg(\overline{B\Big(0; 1 +\frac{\epsilon}{2}  \Big)},\frac{\epsilon}{2} \bigg) \leq \exp \Big(d_{\epsilon} \ \log \Big(\frac{16}{\epsilon} \Big) \Big).
		\end{equation}
	\end{rem}
	
	\subsection{Random Sampling}
	In this section, we deduce the probability estimate concerning the random sampling inequality \eqref{main sampling inq}. We first define independent random variables on $\mathcal{B}_{T,\delta}$ by using random samples. The main tool to derive the probability estimate is the Bernstein inequality which utilizes the upper bounds of the random variables and their variance.
	\begin{lemma}[Bernstein's Inequality \cite{bern}]
		Let $Y_{j},$ $j=1,2,\dots,r$ be a sequence of bounded, independent random variables with $E[Y_{j}]=0,$ $Var[Y_{j}] \leq \sigma^2,$ and $\|Y_{j}\|_{\infty} \leq M$ for $j=1,2,\dots,r.$ Then
		\begin{equation} \label{Bernst}
			P \bigg(\Big|\sum_{j=1}^{r} Y_{j} \Big| \geq \lambda \bigg) \leq 2 \exp \Big(- \frac{\lambda^2}{2 r \sigma^2 + \frac{2}{3} M \lambda }\Big).
		\end{equation}
	\end{lemma}
	Let $S:=\big\{ x_{j}: j\in \NN \big\}$ be i.i.d. random variables distributed uniformly on cube $C_{R}.$ For $f \in \mathcal{B}_{T,\delta},$ we define random variable as follows:
	\begin{equation} \label{random var}
		Z_{j}(f)= |f(x_{j})|^2 x_{j}^{2c-1} - \frac{R^n}{(R^2 -1)^n} \int_{C_{R}} |f(x)|^2 x^{2c-1} dx.
	\end{equation}
	The set $\big\{Z_j(f)\big\}_{j \in \mathbb{N}}$ is a sequence of independent random variables. We also see that
	$$E\big[Z_j(f)\big]= \frac{R^n}{(R^2 -1)^n} \int_{C_{R}} \left[ |f(y)|^2 y^{2c-1} - \frac{R^n}{(R^2 -1)^n} \int_{C_{R}} |f(x)|^2 x^{2c-1} dx  \right]dy=0$$
	\begin{lemma}\label{var}
		Let $f,g \in \mathcal{B}_{T}$ with $\|f\|_{ X_{c}^{2}(\mathbb{R}^{n}_{+})}=\|g\|_{ X_{c}^{2}(\mathbb{R}^{n}_{+})}=1.$ Then the following estimates hold:
		\begin{enumerate}[label=(\roman*)]
			\item $Var\big[Z_{j}(f)\big] \leq \frac{R^{2n}}{(R^2 -1)^n},$
			
			\item $\|Z_{j}(f)\|_{\infty} \leq R^{n},$
			
			\item $\|Z_{j}(f)-Z_{j}(g)\|_{\infty} \leq 2 R^{n} \|f-g\|_{X_{c}^{\infty}(C_{R})}.$
		\end{enumerate}
	\end{lemma}
	
	\begin{proof} Using the definition of variance and (\ref{random var}), we obtain
		\begin{eqnarray*}
			Var\big[Z_{j}(f)\big]&=& E\big[|f(x_{j})|^4 x_{j}^{4c-2}\big]- \big[E \big(|f(x_{j})|^2 x_{j}^{2c-1} \big) \big]^2 \\
			& \leq & E\big[|f(x_{j})|^4 x_{j}^{4c-2}\big] \\
			&=& \frac{R^n}{(R^2 -1)^n} \int_{C_{R}} |f(x)|^4 x^{4c-2} dx \\
			& \leq & \frac{R^{2n}}{(R^2 -1)^n} \|f \|^{2}_{X_{c}^{\infty}(C_{R})} \int_{C_{R}} |f(x)|^2 x^{2c-1} dx \\
			&=& \frac{R^{2n}}{(R^2 -1)^n} \|f \|^{2}_{X_{c}^{\infty}(C_{R})} \|f \|^{2}_{X_{c}^{2}(C_{R})}
		\end{eqnarray*}
		Since $\|f \|^{2}_{X_{c}^{\infty}(C_{R})} \leq \|f\|^{2}_{X_{c}^{2}(\mathbb{R}^{n}_{+})},$ and $\|f\|_{ X_{c}^{2}(\mathbb{R}^{n}_{+})}=1,$ we have
		$$Var\big[Z_{j}(f)\big] \leq \frac{R^{2n}}{(R^2 -1)^n}.$$
		%
		Now since $f \in \mathcal{B}_{T}$ and $\|f\|_{X^{2}_{c}(\mathbb{R}^{n}_{+})}=1,$ we can write
		\begin{eqnarray*}
			\|Z_{j}(f)\|_{\infty}&=& \sup_{x_j \in S} \left||f(x_{j})|^2 x_{j}^{2c-1} - \frac{R^n}{(R^2 -1)^n} \int_{C_{R}}|f(x)|^{2} x^{2c-1} dx  \right| \\
			& \leq & \max \left\{R^n \|f\|_{X_{c}^{\infty}(C_{R})}^{2}, \frac{R^n}{(R^2 -1)^n} \int_{C_{R}} |f(x)|^2 x^{2c-1} dx \right\} \\
			&=& R^n \|f\|_{X_{c}^{\infty}(C_{R})}^{2} \leq R^n. 
		\end{eqnarray*} 
		Similarly for $(iv),$ we have
		\begin{align*}
			&\|Z_{j}(f)-Z_{j}(g)\|_{\infty}\\
			=& \sup_{x_j \in S} \Big| \left( |f(x_{j})|^2 - |g(x_{j})|^2 \right) x_{j}^{2c-1}- \frac{R^n}{(R^2 -1)^n} \int_{C_{R}} (|f(x)|^2-|g(x)|^2) \ x^{2c-1} dx \Big|\\
			\leq & \max \left\{ (|f(x)|-|g(x)|)(\ |f(x)|+|g(x)|) x^{2c-1}, \ \frac{R^n}{(R^2 -1)^n} \int_{C_{R}} (|f(x)|^2 -|g(x)|^2) x^{2c-1} dx  \right\} \\
			\leq& 2 R^{n} \|f-g\|_{X_{c}^{\infty}(C_{R})}.
		\end{align*} 
		This completes the proof.
	\end{proof}
	Using these estimates, we now proceed toward probability estimation. The following result will be helpful in this direction. 
	\begin{thm}\label{prob est}
		Let $\{ x_{j}: j\in \NN \}$ be the i.i.d. random variables drawn uniformly from the cube $C_{R}.$ Then the following holds
		\begin{multline*}
            P \bigg(\sup_{f \in \mathcal{B}^{*}_{T,\delta}} \Big| \frac{1}{r} \sum_{j=1}^{r} Z_{j}(f) \Big| \geq \epsilon \bigg)\leq 2 \exp \bigg[2^n \bigg(\frac{16 T \pi^{-2} R^{2n} 2^{n\left(\frac{n}{2}+2 \right)}}{\epsilon^2}+ (2 T \log R+1)^n \bigg)\\ \log \left(\frac{64 R^n}{\epsilon} \right)- \frac{3 r \epsilon^2 (R^2-1)^n}{ 4 R^n (6 R^{n}+\epsilon(R^2-1)^n)} \bigg].
        \end{multline*}
	\end{thm}
	
	\begin{proof}
		Let $\displaystyle M:=\nn \left(\bb_{T,\delta}^{*},\frac{\epsilon}{4 R^{n}} \right)$ be the covering number for the set $\mathcal{B}^{*}_{T,\delta}$ and $f_{1},f_{2},..., f_{M}$ be the elements of $\mathcal{B}^{*}_{T,\delta}.$ Then for any $f \in \mathcal{B}^{*}_{T,\delta},$ there exists $f_{m},\ 1 \leq m \leq M$ such that $\displaystyle \|f-f_{m}\|_{X^{\infty}_c(C_{R})} < \frac{\epsilon}{4 R^{n}}.$ From (iii) of Lemma \ref{var}, we have
		$$ \left|\frac{1}{r} \sum_{j=1}^{r} \left(Z_{j}(f)-Z_{j}(f_{m}) \right) \right| \leq 2 R^{n} \|f-f_{m}\|_{X^{\infty}_c(C_{R})} \leq \frac{\epsilon}{2}.$$ 
		In view of (\ref{Bernst}), for given $f_m,\ 1 \leq m \leq M,$ we have
		$$ P \bigg( \bigg| \frac{1}{r} \sum_{j=1}^{r} Z_{j}(f_{m}) \bigg| \geq \frac{\epsilon}{2} \bigg) \leq  2 \exp \left(- \frac{3 r \epsilon^2 (R^2-1)^n}{ 4 R^n (6 R^{n}+\epsilon(R^2-1)^n)} \right).$$ For fixed $m,$ we have
		\begin{eqnarray*}
			P \bigg( \sup_{ \left\{f: \|f-f_{m}\|_{X^{\infty}_c(C_{R})} \leq \frac{\epsilon}{4 R^{n}}\right\} } \bigg|\frac{1}{r} \sum_{j=1}^{r} Z_{j}(f) \bigg| \geq \epsilon \bigg) & \leq & P \bigg( \bigg| \frac{1}{r} \sum_{j=1}^{r} Z_{j}(f_{m}) \bigg|\geq \frac{\epsilon}{2} \bigg) \\
			& \leq &  2 \exp \left(-\frac{3 r \epsilon^2 (R^2-1)^n}{ 4 R^n (6 R^{n}+\epsilon(R^2-1)^n)} \right).
		\end{eqnarray*}
		Since $M$ is the covering number for $\mathcal{B}^{*}_{T,\delta},$ i.e.,
		$$\mathcal{B}^{*}_{T,\delta} \subset \bigcup_{m=1}^{M} \left\{ f: \|f-f_{m}\|_{X^{\infty}_c(C_{R})} \leq \frac{\epsilon}{4 R^{n}} \right\},$$
		we obtain 
		\begin{eqnarray*}
			P \bigg(\sup_{f \in \mathcal{B}^{*}_{T,\delta}} \Big| \frac{1}{r} \sum_{j=1}^{r} Z_{j}(f) \Big| \geq \epsilon \bigg) & \leq & \sum_{m=1}^{M} P \bigg( \sup_{ \left\{f: \|f-f_{m}\|_{X^{\infty}_c(C_{R})} \leq \frac{\epsilon}{4 R^{n}} \right\}} \left|\frac{1}{r}\sum_{j=1}^{r} Z_{j}(f) \right| \geq \epsilon \bigg) \\
			& \leq & 2 M \exp \bigg(-\frac{3 r \epsilon^2 (R^2-1)^n}{ 4 R^n (6 R^{n}+\epsilon(R^2-1)^n)} \bigg).
		\end{eqnarray*}
		Using \eqref{covering}, we finally get
        \begin{multline*}
            P \bigg(\sup_{f \in \mathcal{B}^{*}_{T,\delta}} \Big| \frac{1}{r} \sum_{j=1}^{r} Z_{j}(f) \Big| \geq \epsilon \bigg)\leq 2 \exp \bigg[2^n \bigg(\frac{16 T^n \pi^{-2n} R^{2n} 2^{n\left(\frac{n}{2}+2 \right)}}{\epsilon^2}+ (2 T \log R+1)^n \bigg)\\ \log \left(\frac{64 R^n}{\epsilon} \right)- \frac{3 r \epsilon^2 (R^2-1)^n}{ 4 R^n (6 R^{n}+\epsilon(R^2-1)^n)} \bigg].
        \end{multline*}
	
	\end{proof}
	
	We now prove the following main result of the paper.
	\begin{thm} \label{main}
		Let $\{ x_j: j\in \NN \}$ be the independent and identically distributed random variables that are uniformly distributed over $C_R.$ For $0< \mu < 1-\delta,$ the following inequality
		\begin{equation} \label{main sampling inq}
			\frac{R^{n-1}(1-\delta - \mu)}{(R^2-1)^n} \|f\|_{X_c^{2}(\mathbb{R}^{n}_{+})}^2 \leq \frac{1}{r} \sum_{j=1}^{r} |f(x_{j})|^2 x_{j}^{2c} \leq \frac{R^{n+1} (1+\mu)}{(R^2-1)^n} \|f\|_{X_c^{2}(\mathbb{R}^{n}_{+})}^2,
		\end{equation}
		holds for every $f \in \bb_{T,\delta}$ with probability at least $\displaystyle 1- 2 \beta\exp \Big(-\frac{3r\mu^2}{4 (R^2-1)^n (6+\mu)} \Big),$ where $\displaystyle \beta:=\exp\bigg( 2^n \Big(\frac{16 T^n \pi^{-2n} (R^2-1)^{2n} \ 2^{n\left(\frac{n}{2}+2 \right)}}{\mu^2}+ (2 T \log R+1)^n \Big) \log \Big(\frac{64 (R^2-1)^n}{\mu} \Big) \bigg).$
	\end{thm}
	
	\begin{proof} Let $\{ x_j:j\in J\}$ be the uniformly distributed samples taken from $C_R$ and define the following event
		$$ \mathcal{E}= \left \{\sup_{f \in \mathcal{B}^{*}_{T,\delta}}  \left | \frac{1}{r} \sum_{j=1}^{r} Z_{j}(f) \right| \leq \frac{\mu R^n}{(R^2 -1)^n} \right \}.$$
		The event $\mathcal{E}$ is equivalent to the event
		$$ \bigg| \sum_{j=1}^{r} |f(x_{j})|^2 x_{j}^{2c-1} - \frac{r R^n}{(R^2 -1)^n} \int_{C_{R}} |f(x)|^2 x^{2c-1} dx \bigg|  \leq \frac{r \mu R^n}{(R^2 -1)^n}, \qquad f\in \bb_{T,\delta}^{*}.$$
		This implies that for all $f\in \bb_{T,\delta}^{*}$
		\begin{align*}
			\frac{rR^n}{(R^2-1)^n} \int_{C_{R}} |f(x)|^2 x^{2c-1}\,dx - \frac{r \mu R^n}{(R^2-1)^n} &\leq \sum_{j=1}^{r} |f(x_{j})|^2 x_{j}^{2c-1} \\
			&\leq \frac{rR^n}{(R^2-1)^n} \int_{C_{R}} |f(x)|^2 x^{2c-1}\,dx + \frac{r \mu R^n}{(R^2-1)^n}\\
            \frac{R^{n-1}(1-\delta - \mu)}{(R^2-1)^n} \leq \frac{1}{r} \sum_{j=1}^{r} |f(x_{j})|^2 x_{j}^{2c} &\leq \frac{R^{n+1}}{(R^2-1)^n}(1+\mu).
		\end{align*}
		It can be observed that the event $\mathcal{E}$ leads to the desired sampling inequality \eqref{main sampling inq}. From Theorem \ref{prob est}, the probability that the random samples $\{ x_{j}: j=1,2,\dots,r \}$ satisfy the above sampling inequality, is given by
		\begin{eqnarray*}
			P(\mathcal{E})&=& 1-P(\mathcal{E}^c)\\
			& \geq & 1-2 \exp \left(-r \alpha+ \beta \right)
		\end{eqnarray*}
		where $\displaystyle \alpha:= \frac{3 \mu^2}{4 (R^2-1)^n (6+\mu)},\ $ and \\
		$\displaystyle \beta:= 2^n \bigg(\frac{16 \pi^{-2n} T^n (R^2-1)^{2n} \ 2^{n\left(\frac{n}{2}+2 \right)}}{\mu^2}+ (2 T \log R+1)^n \bigg) \log \bigg(\frac{64 (R^2-1)^n}{\mu} \bigg).$
	\end{proof}
	
	\begin{rem}
		From the above estimate of $P(\mathcal{E}),$ it is easy to see that the probability approaches to $1$ for sufficiently large sample size $r.$ Moreover, the sampling inequality \eqref{main sampling inq} is true for large value of $r$ such that
		$r > \frac{\beta}{\alpha}.$ 
	\end{rem}

	\section*{Acknowledgment}
	S. Bajpeyi and S. Sivananthan gratefully thank financial assistance from the Department of Science and Technology, Government of India, via project no. CRG/2019/002412.

	\end{document}